\documentclass[11pt,a4paper]{amsart}

\usepackage{mathrsfs}
\usepackage{amsmath, amsthm, amsfonts}
\usepackage[percent]{overpic}                               

\usepackage{color}
\usepackage{MnSymbol}
\usepackage{fullpage}
\usepackage{bbm}

\usepackage{xcolor}
\usepackage{graphicx}
\usepackage[top=2cm,bottom=2cm,left=2cm,right=2cm,a4paper]{geometry}
\usepackage{thm-restate}

\usepackage[all]{xy}

\usepackage{mathrsfs}
\usepackage{tikz-cd}

\usepackage{thmtools}
\usepackage{thm-restate}

\usepackage{hyperref}

\usepackage{cleveref}

\declaretheorem[name=Theorem,numberwithin=section]{thm}

\newtheorem*{thm*}{Theorem}
\newtheorem{cor}[thm]{Corollary}
\newtheorem{prop}[thm]{Proposition}

\theoremstyle{definition}

\newtheorem*{ack}{Acknowledgements}

\theoremstyle{remark}

\newcommand{\QQ}{\mathbb{Q}}

\newcommand{\HH}{\mathcal{H}}

\newcommand{\CC}{\mathbb{C}}

\newcommand{\OO}{\mathcal{O}}

\newcommand{\ZZ}{\mathbb{Z}}

\newcommand{\PP}{\mathbb{P}}

\newcommand{\Pic}{\text{Pic}}
\newcommand{\codim}{\text{codim}}
\newcommand{\Mg}{\mathcal{M}_g}
\newcommand{\Mgn}{\mathcal{M}_{g,n}}
\newcommand{\Mgb}{\overline{\mathcal{M}}_g}
\newcommand{\Mgnb}{\overline{\mathcal{M}}_{g,n}}

\newcommand{\Hdg}{\mathcal{H}_{g,d}}
\newcommand{\Hdgf}{\mathcal{H}^+_{g,d}}
\newcommand{\Hdgb}{\overline{\mathcal{H}}_{g,d}^{\text{ad}}}
\newcommand{\Hdgfb}{\overline{\mathcal{H}}^{+\text{ad}}_{g,d}}
\newcommand{\Hdgbv}{\overline{\mathcal{H}}_{g,d}}
\newcommand{\Hdgfbv}{\overline{\mathcal{H}}^{+}_{g,d}}

\newcommand{\ram}{\text{ram}}
\newcommand{\Kod}{\text{Kod}}

\newcommand{\M}[2]{\mathcal{M}_{{#1}, {#2}}}

\newcommand{\Mbar}[2]{\overline{\mathcal{M}}_{{#1}, {#2}}}

\newcommand{\PHt}{\widetilde{\PP\HH}_X^{g,d}}
\newcommand{\PHb}{\overline{\PP\HH}_X^{g,d}}

\newtheorem{Thm*}{Theorem*}
\theoremstyle{definition}

\title{The Hurwitz space Picard rank conjecture for $d>g-1$ 
\\ 
}
\author{Scott Mullane}
\date{\today}

\AtEndDocument{\bigskip{\footnotesize%
\textsc{Scott Mullane, Humboldt-Universit\"at zu Berlin, Institut f\"ur Mathematik,  Unter den Linden 6, 10099 Berlin, Germany} \email{{\tt scott.mullane@hu-berlin.de}}
\par
}}

\begin{document}
\thispagestyle{empty}

\maketitle

\begin{abstract}
We show the simple Hurwitz space $\mathcal{H}_{g,d}$ has trivial rational Picard group for $d>g-1$. 
\end{abstract}

\setcounter{tocdepth}{1}


\section{Introduction}

The Hurwitz space $\Hdg$, parameterising isomorphism classes of simply branched genus $g$, degree $d$ covers of the rational line has resided at the intersection of geometry, algebra, topology and analysis for over a century and a half. From the perspective of algebraic geometry, Hurwitz spaces have been a major tool in understanding important aspects of the moduli space of curves $\Mg$, including the dimension~\cite{Riemann}, irreducibility~\cite{Clebsch},\cite{Hurwitz}, dimension of proper subvarieties~\cite{Diaz}, and the birational classification~\cite{HarrisMumford}. Despite this, some fundamental questions on the global geometry of $\Hdg$ persist. In this paper we consider the questions of the Picard group and the birational classification problem for $\Hdg$.

It has long been conjectured that the rational Picard group $\Pic_\QQ(\Hdg)$ is trivial and was known in the range $d\leq5$ and $d>2g-2$. The case of $d=2$ was proven by Cornalba and Harris~\cite{CornalbaHarris}, $d=3$ by Stankova-Frenkel~\cite{SF} and $d=4,5$ by Deopurkar and Patel~\cite{DP}. When $d>2g-2$ the map $\Hdg\longrightarrow\mathcal{M}_{g}$ is a fibration over $\mathcal{M}_g$ and the result follows from Harer's theorem~\cite{Harer} that the rational Picard group of $\Mg$ has rank one (see~\cite{Moch} or \cite{DiazEdidin}). 

For reasons that will become clear when we discuss our methods, we also consider the finite cover $\Hdgf$ of $\Hdg$ obtained by marking one branch point. The main result of this paper is the following.

\begin{restatable}{thm}{PicRank}
\label{thm:PicRank}
The rational Picard groups $\Pic_{\QQ}(\Hdgf)$ for $d>g$ and $\Pic_{\QQ}(\Hdg)$ for $d>g-1$ are trivial.
\end{restatable}

The compactification we construct to prove Theorem~\ref{thm:PicRank} carries natural rulings leading to the following observation.

\begin{restatable}{lem}{KodDim}
\label{thm:KodDim}
$\Hdgf$ and $\Hdg$ are uniruled when $d>g+1$. 
\end{restatable}

Further, $\Hdgf$ and $\Hdg$ are not unirational in this range for $g\geq10$ due to the natural dominant map to $J_{g,d}$ the universal genus $g$ degree $d$ Jacobian which is of general type for $d>g+1$ and $g\geq 10$. 

Our strategy to prove Theorem~\ref{thm:PicRank} is to consider $\Hdgf$ as an open dense subvariety of a projectivised stratum of exact differentials of fixed signature and obtain a compactification by taking the closure inside the projectivisation of a twisted Hodge bundle of stable differentials.

From any degree $d$ simply branched cover $f:C\longrightarrow\PP^1$ with a branch point at $\infty$ we obtain a meromorphic exact differential $f^*dz$ on $C$ with signature of zeros and poles $(-3,-2^{d-2},1^{2g+2d-3})$ where the poles occur at the points in $f^{-1}(\infty)$ and the zeros at the remaining ramification points of $f$. Similarly, given any exact differential of this signature, by integrating the differential we obtain a degree $d$ cover of the rational line simply branched over $\infty$ with possibly non-simple branching elsewhere, though as the zeros of the differential are simple the ramification will be simple, that is, the branching profile is always of type $(2^k,1^{d-2k})$ for some $k$.

In \S\ref{sec:proj} we show this identifies $\Hdgf$ as an open subvariety inside the projectivisation of a twisted Hodge bundle of rank $2d+g$ over 
$$\M{g}{1+[d-2]}:=\M{g}{d-1}/\mathfrak{S}_{d-2}$$
with fibres over $[C,p_1,p_2+\dots+p_{d-1}]$ identified with 
$$H^0(\omega_C(3p_1+2(p_1+\dots+p_{d-1}))).$$
Taking the closure we obtain the partial compactification $\PHt$ of $\Hdgf$. Considering the closure in the natural extension of the twisted Hodge bundle to stable differentials over $\Mbar{g}{1+[d-2]}$ we obtain the compactification $\PHb$.

For $d>g+1$, Proposition~\ref{prop:proj} shows $\PHt\setminus\widetilde{X}$  is a projective bundle over  $\M{g}{1+[d-2]}\setminus X$ where $\widetilde{X}$ and $X$ are subvarieties of codimension at least two in $\PHt$ and $\M{g}{1+[d-2]}$ respectively. Hence Lemma~\ref{thm:KodDim} holds and in these cases we also obtain the rank of the rational Chow group
\begin{equation}\label{eq:rank}
\rho(\PHt)=\rho(\M{g}{1+[d-2]})+1=4.
\end{equation}

For $d=g+1$, Proposition~\ref{prop:d=g+1} gives a proper birational morphism 
$$\pi:\PHt\longrightarrow \M{g}{1+[d-2]}$$
and Lemma~\ref{thm:KodDimToo} gives the implications to the Kodaira dimension of $\HH_{g,g+1}^+$ that follow from results on $\M{g}{1+[g-1]}$ in $g=10$ by Barros and the author~\cite{BM} and $g\geq11$ by Farkas and Verra~\cite{FV},\cite{FarkasVerraTheta}. Further, as the exceptional locus is irreducible we obtain bounds on the rank of the rational Chow group
$$\rho(\M{g}{1+[g-1]}) \leq \rho(\widetilde{\PP\HH}_X^{g,g+1})\leq \rho(\M{g}{1+[g-1]})+1.$$
The results of \S 4 on the linear independence of the boundary components then imply that Equation~\ref{eq:rank} also holds in this case.

We begin \S\ref{sec:PicRank} by completing the remaining technical task of showing that for $d>g$ the rational Chow and hence Picard rank of $\PHt$ is completely accounted for by components of the boundary $\PHt\setminus \Hdgf$. We consider the well known compactifications $\Hdgfb$ and $\Hdgb$ by admissible covers~\cite{HarrisMumford} and their normalisations $\Hdgfbv$ and $\Hdgbv$ the spaces of twisted stable maps~\cite{ACV}. The identification of $\Hdgf$ as a open Zariski dense subvariety of $\PHt$ by pulling back a differential from the target curve to the source curve extends to a birational morphism
$$\tau:\Hdgfbv\setminus Y\longrightarrow  \PHt$$
for $Y$ a certain collection of boundary divisors of $\Hdgfbv$.

Hence any non-trivial relation in the four boundary components of $\PHt$ will pullback to give a non-trivial relation in the boundary components of $\Hdgfbv$. It is not known that the boundary components of $\Hdgfbv$ are linearly independent, however, Patel~\cite{Patel} showed the irreducible boundary components of $\Hdgbv$ are independent. Hence the pushforward of any non-trivial relation in the boundary components under the natural finite forgetful morphism
$$\varphi:\Hdgfbv\longrightarrow \Hdgbv$$
must be the trivial relation. With the addition of a number of test curves in $\Hdgfbv$ we show that the pullback of any relation in the boundary components of $\PHt$ must be trivial and hence the components of the boundary of $\PHt$ are linearly independent.

For $d=g$, the map 
$$\pi:\PHt\longrightarrow \M{g}{1+[d-2]}$$
is no longer dominant and we must use a different approach. Consider the finite branch morphism 
$$f_{\text{br}}:\Hdgbv\longrightarrow \Mbar{0}{[2g+2d-2]}$$
and let $Y$ be the collection of all boundary divisors other than the three irreducible components of $f_{\text{br}}^*\delta_{[2]}$ where the stable model of the general source curve is a smooth curve (see Figure~\ref{Divisors2}). Then we obtain an isomorphism
$$\begin{array}{rcl}
\overline{\HH}_{g,g}\setminus \{Y\cup\widetilde{X}\}&\longrightarrow&\mathcal{C}_{g,g-2}\setminus X\\
\left[\pi:C\to\PP^1\right]&\mapsto&\left[C,\omega_C\otimes\pi^*\OO(-1)\right]\\ 
\end{array}$$
where $\widetilde{X}$ and $X$ are the loci of indeterminacy of the associated rational map and the rational inverse respectively (see \S~\ref{sec:PicRank} for details). As $\codim(\widetilde{X})=\codim(X)=2$ we obtain the rank of the rational Chow group
$$\rho(\overline{\HH}_{g,g}\setminus Y)=\rho(\mathcal{C}_{g,g-2})=3.$$
The boundary components of $\Hdgbv$ are linearly independent by Patel~\cite{Patel} which completes the final case of Theorem~\ref{thm:PicRank}.

The paper is organised as follows. In \S\ref{sec:proj} we define and discuss the properties of the compactifications $\PHb$ and $\PHt$ of $\Hdgf$. In \S\ref{sec:kod} we prove the results on the Kodaira dimension of Hurwitz space. In \S\ref{sec:PicRank} we show that for $d>g$ the rational Chow rank of $\PHt$ is completely accounted for by the components of the boundary $\PHt\setminus\Hdgf$, while for $d=g$ the same holds for $\Hdgbv\setminus Y$, hence proving Theorem~\ref{thm:PicRank}.

\begin{ack}
I am grateful to Ignacio Barros, Dawei Chen and Martin M\"{o}ller for many useful discussions. The author was supported by the Alexander von Humboldt Foundation during the preparation of this article. 
\end{ack}

\section{The twisted Hodge bundle and Hurwitz space}\label{sec:proj}
The space of holomorphic differentials on genus $g$ curves forms a vector bundle $\HH$ of rank $g$ over $\Mg$ which extends to a rank $g$ vector bundle $\overline{\HH}$ over $\Mgb$, the Deligne-Mumford moduli space of stable curves. If
$$\pi:\mathcal{C}\longrightarrow \Mgb$$
is the universal family, then $\overline{\HH}$ is defined as the pushforward of the relative dualising sheaf
$$\overline{\HH}:=\pi_*\omega_{\mathcal{C}/\Mgb}.$$
The fibre of $\overline{\HH}$ over a stable nodal curve $C$ is $H^0(C,\omega_C)$, the global sections of the dualising sheaf $\omega_C$ (see~\cite[Section 3A]{HarrisMorrison}).

Similarly, for any $\mu=(m_1,\dots,m_n)$ with $m_i\leq 0$   
$$\HH(\mu,\overline{\{1^{2g-2-\sum m_i}\}})$$
is a rank $g-1-\sum m_i$ vector bundle over $\Mgn$ where the fibre over $[C,p_1,\dots,p_n]\in \Mgn$ is given by 
$$H^0(C,\omega_C(-\sum m_i p_i))$$
which is naturally identified with the $g-1-\sum m_i$ dimensional vector space of differentials on $C$ with poles of order at most $-m_i$ at $p_i$ for $m_i<0$ and no other poles. Here the notation of the brackets is indicating that these simple zeros in the twisted differentials are unordered, while the line above indicates that these simple zeros can collide with each other and with the marked points in $\Mgn$. We let 
$$\HH(\mu,\{1^{2g-2-\sum m_i}\})$$
denote the open subvariety where the zeros do not collide with each other or the marked points. 

This bundle extends in the same way to a rank $g-1-\sum m_i$ vector bundle over $\Mgnb$. If
$$\pi:\mathcal{C}\longrightarrow \Mgnb$$
is the universal family, define
$$\overline{\HH}(\mu,\{1^{2g-2-\sum m_i}\}):=\pi_*\omega_{\mathcal{C}/\Mgnb}\big(-\sum_{i=1}^n m_i\mathcal{Z}_i\big)$$ 
the bundle over $\Mgnb$, where $\mathcal{Z}_i$ denotes the image of the section of the universal family $\pi$ given by the $i$th marked point.

A natural subvariety of the twisted Hodge bundle is obtained by restricting to exact differentials, that is, meromophic differentials $\eta$ on a smooth curve $C$ of the required signature such that if $P$ is the set of poles of $\eta$, then $\int_\gamma\eta=0$ for any $\gamma\in H_1(C\setminus P,\ZZ)$. We denote this subvariety by 
$$\HH_X(\mu,\overline{\{1^{2g-2-\sum m_i}\}}).$$ 
Fix $\mu=(-3,-2^{d-2})$ and consider $\HH_X(\mu,{\{1^{2g+2d-3}\}})$. Integrating the differential gives a degree $d$ cover of the rational line, in which, the image of the marked points in $\M{g}{d-1}$ are all $\infty$. 

Similarly, given any simply branched degree $d$ cover of the rational line by a genus $g$ curve, pulling back a differential on the rational curve with a double pole at one of the branch points and no other poles or zeros gives an exact differential of this type with the double poles unordered. We let
$$\PP\HH_X^{g,d}:=\PP\HH_X(\mu,{\{1^{2g+2d-3}\}})/\mathfrak{S}_{d-2}$$
where $\mathfrak{S}_{d-2}$ acts by permutation on the labelling of the marked poles of order $2$. The above identification gives an isomorphism between an open Zariski dense subvariety of $\PP\HH_X^{g,d}$ and a pointed version of the simple Hurwitz space $\Hdgf$. Specifically, if $Y$ is the subvariety of $\PP\HH_X^{g,d}$ defined by the condition that the degree $d$ cover obtained by integrating the differential does not have simple branching, then 
$$\PP\HH_X^{g,d}\setminus Y\cong \Hdgf$$
where $\Hdgf$ is the Hurwitz space of genus $g$, degree $d$ covers of the rational line with simple branching and one branch point marked. As the zeros of the differential remain distinct any non-simple branch profile will be of the type $(2^k,1^{d-2k})$ for some $k$. We denote by $\overline{\PP\HH}_X^{g,d}$, the closure of $\PP\HH_X^{g,d}$ inside the bundle
$$ \PP\overline{\HH}(\mu,\overline{{\{1^{2g+2d-3}\}}})/\mathfrak{S}_{d-2}$$
over 
$$\overline{\mathcal{M}}_{g,1+[d-2]}:=\Mbar{g}{d-1}/\mathfrak{S}_{d-2}.$$ 
An exact differential $\eta$ on a curve $C$ satisfies $\int_\gamma\eta=0$ for $\gamma\in H_1(C\setminus P)$ where $P$ is the set of points appearing as poles of $\eta$. Hence integrating $\eta$ gives a cover of $C\to\PP^1$ with the fibre above $\infty$ given by the pole divisor of $\eta$. In the other direction, any cover $C\to\PP^1$ with simple branching over $\infty$ will give an exact differential with the given pole orders at the points in this branch fibre by pulling back $dz$ on $\PP^1$. Hence the fibre over $[C,p_1,p_2+\dots+p_{d-1}]\in\mathcal{M}_{g,1+[d-2]}$ is non-empty when we can construct such a map
$$\dim H^0(2p_1+p_2+\dots+p_{d-1})\geq 2$$
and in this case, by Riemann-Roch the fibre dimension is equal to
$$\dim H^0(2p_1+p_2+\dots+p_{d-1})-2=d-g-1+\dim H^0(\omega_C(-2p_1-p_2-\dots-p_{d-1})).$$
Hence 
$$\pi:\overline{\PP\HH}_X^{g,d}\longrightarrow \overline{\mathcal{M}}_{g,1+[d-2]}$$
is a dominant morphism for $d>g$. For a complete description (beyond the scope of our discussion) of the degeneration of pointed differentials to stable nodal curves see~\cite{BCGGM} and the recent smooth compactification of the strata of differentials~\cite{BCGGM2}. For a discussion of the limit of zero residue conditions and applications to the irreducibility of certain non-simple Hurwitz spaces see~\cite{Mullane}.

We define $\PHt$ to be the partial compactification of $\Hdgf$ obtained by taking the closure inside the Hodge bundle over $\M{g}{1+[d-1]}$, or equivalently, restricting $\overline{\PP\HH}_X^{g,d}$ to the pre-image of $\M{g}{1+[d-1]}$. This is simply the projectivisation of 
$$\HH_X(-3,\{-2^{d-2}\},\overline{\{1^{2g+2d-3}  \}}).$$ 
Define $X$ to be the locus of $[C,p_1,p_2+\dots+p_{d-1}]\in\mathcal{M}_{g,1+[d-2]}$ such that 
$$\dim H^0(\omega_C(-2p_1-p_2-\dots-p_{d-1}))>0$$
and $\widetilde{X}$ to be the pre-image in $\PHt$.

The following proposition shows that when $\pi$ has positive dimensional fibres, $\widetilde{X}$ has codimension at least two and $\PHt\setminus\widetilde{X}$ is a projective bundle.

\begin{prop}\label{prop:proj}
Fix $d>g+1$, then $X$ and $\widetilde{X}$ have codimension at least two in $\M{g}{1+[d-2]}$ and $\PHt$ respectively and $\PHt\setminus \widetilde{X}$ is a projective bundle over $\M{g}{1+[d-2]}\setminus X$.
\end{prop}

\begin{proof}
For $[C,p_1,p_2+\dots+p_{d-1}]$ in $\M{g}{1+[d-2]}$, the exact differentials form a linear subspace in 
$$H^0(\omega_C(3p_1+2(p_2+\dots+p_{d-1})).$$ 
As discussed above, the fibre dimension is given by
$$d-g-1+\dim H^0(\omega_C(-2p_1-p_2-\dots-p_{d-1})).$$ 
Our task is to show that $X$ and $\widetilde{X}$ have codimension at least two in $\M{g}{1+[d-2]}$ and $\PHt$ respectively.  

Consider the pointed version of the subvariety of exact differentials in the twisted Hodge bundle 
$${\PP\HH}_X(-3,-2^{d-2},\overline{\{1^{2g+2d-3}\}}).$$ 
Here the double poles are ordered and the bar indicates that we consider the partial compactification where the simple zeros can collide with each other and the poles, though the underlying pointed curve in $\M{g}{d-1}$ remains smooth. For $r\geq 1$, let $X_r$ be the locus in $\M{g}{d-1}$ with fibre dimension at least $d-g-1+r$, hence the locus such that 
$$\dim H^0(\omega_C(-2p_1-p_2-\dots-p_{d-1}))\geq r.$$
For $d>g+1$, because $\pi$ is dominant and $\Hdgf$ is irreducible, the pre-image of these loci have at most codimension one. Further, there exist such a locus with pre-image codimension one in ${\PP\HH}_X(-3,-2^{d-2},\overline{\{1^{2g+2d-3}\}})$ if and only if there exists an $r$ such that $\codim(X_r)=r+1$. We show no such locus exists.

Assume $\codim(X_r)=r+1$ and let 
$$\varphi_j:\Mbar{g}{d-1}\longrightarrow\Mbar{g}{d-2}$$ 
be the morphism forgetting the $j$th point for $j=2,\dots,d-1$. We denote the closure of ${X}_r$ in $\Mbar{g}{d-1}$ by $\overline{X}_r$. Then $\codim(\overline{X}_r)=r+1$ in $\Mbar{g}{d-1}$. Hence $\codim(\pi_j(\overline{X}_r))=r$ or $r+1$. However, for any $[C,p_1,\dots,p_{d-1}]$ in $X_r$ we observe
$$\dim H^0(\omega_C(-2p_1-p_2-\dots-0p_j-\dots-p_{d-1}))\geq r$$
and if $\codim(\pi_j(\overline{X}_r))=r$ in $\Mbar{g}{d-2}$ then the pre-image of $\pi_j(\overline{X}_r)$ provides a full dimensional component of $\widetilde{\PP\HH}_X^{d-1,g}$ contradicting the irreducibility of $\HH_{d-1,g}$.

Hence if $\codim(X_r)=r+1$ in $\Mbar{g}{d-1}$, then $\codim(\pi_j(\overline{X}_r))=r+1$ in $\Mbar{g}{d-2}$ for $j=2,\dots,d-1$ which implies that the general fibre has dimension one, that is, the forgotten point $p_j$ moves freely in $\overline{X}_r$. Hence for any $[C,p_1,\dots,p_{d-1}]$ in $\overline{X}_r$ the $d-2$ points $p_j$ for $j=2,\dots,d-1$ move freely in the $g-1$ dimensional canonical system $|\omega_C|$ which provides a contradiction for $d>g+1$ as $d-2>g-1$. Hence $\codim(X_r)\geq r+2$.
\end{proof}

In the case that $d=g+1$ the dimensions coincide and $\pi$ is a birational morphism with irreducible exceptional locus.
\begin{prop}\label{prop:d=g+1}
Fix $d=g+1$, then 
$$\pi:\widetilde{\PP\HH}_X^{g,g+1}\longrightarrow\M{g}{1+[g-1]}$$ 
is a proper birational morphism that has irreducible exceptional locus in codimension one $E$ that is contracted to the codimension two locus of $[C,p_1,p_2+\dots+p_{g}]$ such that 
$$\dim H^0(\omega_C(-2p_1-p_2-\dots-p_{g}))\geq 1.$$
\end{prop}

\begin{proof}
The expected dimension of $H^0(2p_1+p_2+\dots+p_{g})$ is $2$, giving a unique pre-image under $\pi$. The existence of $X_r$ for $r\geq 2$ as defined above with $\codim(X_r)=r+1$ would again contradict the irreducibility of $\HH_{g,g}$. The irreducibility of this locus then follows from the irreducibility of the image in $\overline{\mathcal{M}}_{g,1+[g-1]}$. This locus is the image of the projectivised stratum 
$$\PP \HH(2,\{1^{g-1}\},\overline{\{1^{g-3}\}})$$ 
in $\M{g}{1+[g-1]}$ after forgetting the last $g-3$ points.
However, $\HH(2,1^{2g-4})$ is irreducible~\cite{KontsevichZorich}.
\end{proof}

\begin{prop}\label{prop:d=g}
Fix $d=g$, then $\overline{\PP\HH}_X^{g,g}$ is birational to $\overline{\PP\HH}(2,\{1^{g-2}\},\{1^{g-2}\})$.
\end{prop}

\begin{proof}
Consider the forgetful morphism from both varieties to $\Mbar{g}{1+[g-2]}$. The fibre of a generic point in the image of each is a single point. The image of each is equal to the closure of the locus of $[C,p_1,p_2+\dots+p_{g-1}]$ in $\M{g}{1+[g-2]}$ such that 
$$\dim H^0(\omega_C(-2p_1-p_2-\dots-p_{g-2})\geq 1.$$
\end{proof}

Proposition~\ref{prop:proj} also provides the rank of the rational Chow group of $\PHt$ when $d>g+1$.
\begin{cor}\label{cor:ChowRankd>g+1}
Fix $d>g+1$, then the following relation on the rank of the rational Chow groups holds
$$\rho(\PHt)=\rho(\M{g}{1+[d-2]})+1.$$
\end{cor}

\begin{proof}
By the excision sequence as $\codim(\widetilde{X})\geq2$ we obtain
$$A^1_\QQ(\PHt)\cong A^1_\QQ(\PHt\setminus\widetilde{X})$$
and as $\codim(X)\geq2$
$$A^1_\QQ(\M{g}{1+[d-2]})\cong A^1_\QQ(\M{g}{1+[d-2]}\setminus X).$$
Further, the rational Chow group of a projective bundle is generated by the $\OO(-1)$ line bundle class of the projective bundle and the pullback of the Chow group of the base.  Hence Proposition~\ref{prop:proj} gives the result for $d>g+1$. 
\end{proof}

In the case that $d=g+1$, Proposition~\ref{prop:d=g+1} gives bounds on the rank of the rational Chow group.
\begin{cor}\label{cor:ChowRankd=g+1}
Fix $d=g+1$, then the following relation on the rank of the rational Chow groups holds
$$\rho(\M{g}{1+[g-1]}) \leq \rho(\widetilde{\PP\HH}_X^{g,g+1})\leq \rho(\M{g}{1+[g-1]})+1.$$
\end{cor}

\begin{proof}
For $d=g+1$, observe that 
$$A^1_{\QQ}(\widetilde{\PP\HH}_X^{g,g+1}\setminus E)\cong A^1_{\QQ}(\M{g}{1+[g-1]}\setminus \pi(E))\cong A^1_{\QQ}(\M{g}{1+[g-1]})$$
as $\codim(\pi(E))=2$. Hence by excision we obtain the exact sequence
$$\QQ[E]\longrightarrow A^1_\QQ(\widetilde{\PP\HH}_X^{g,g+1})\longrightarrow A^1_\QQ(\M{g}{1+[g-1]})\longrightarrow0.$$
\end{proof}

\section{On the Kodaira dimension of Hurwitz space}\label{sec:kod}

From the last section we observe the following results on the Kodaira dimension of Hurwitz space.

\KodDim*

\begin{proof}
In this case all fibres of
$$\pi:\PHt\longrightarrow\M{g}{1+[d-2]}$$ 
are positive dimensional projective spaces, hence providing rulings. The finite morphism
$$\varphi:\Hdgf\longrightarrow \Hdg$$
implies the rulings push forward and the result also holds for $\Hdg$.
\end{proof}

\begin{restatable}{lem}{KodDimToo}
\label{thm:KodDimToo}
For $d=g+1$, 
$$\Kod(\Hdgf)=\begin{cases}0&\text{ for $g=10$}\\
19& \text{ for $g=11$}\\
\text{maximal}&\text{ for $g\geq12$,}  \end{cases}
\hspace{0.7cm}$$
\end{restatable}

\begin{proof}
The Kodaira dimension is a birational invariant. Hence the first result follows from Proposition~\ref{prop:d=g+1} and the Kodaira dimension of $\Mbar{g}{1+[g-1]}$ due to Barros and the author~\cite{BM} for $g=10$ who showed that 
$$\Kod(\Mbar{10}{10}/H)=0$$
for all subgroups $H$ of $\mathfrak{S}_{10}$. Farkas and Verra~\cite{FV} show that 
$$\Kod(\Mbar{11}{11}/H)=19$$
for all subgroups $H$ of $\mathfrak{S}_{11}$. Further, Farkas and Verra~\cite{FarkasVerraTheta} show $\Kod(\Mbar{g}{[g-1]})$ is maximal for $g\geq 12$. Hence the final result follows from the forgetful morphism 
$$\Mbar{g}{1+[g-1]}\longrightarrow \Mbar{g}{[g-1]}$$
for $g\geq 12$ that maps $\Mbar{g}{1+[g-1]}$ dominantly to a variety of general type with fibres of general type. 

\end{proof}
Note that for $g\geq 12$, the variety $\HH_{g,g+1}^+$ is of general type and provides a finite cover of $\HH_{g,g+1}$, which is of intermediate type with $\Kod(\HH_{g,g+1})=3g-3$. The implications of Proposition~\ref{prop:d=g} to the Kodaira dimension of $\Hdg$ were already known.

\section{On the Picard group of Hurwitz space}\label{sec:PicRank}

\subsection{The admissible covers compactification of $\Hdgf$}
In \S\ref{sec:proj} we constructed $\overline{\PP\HH}_X^{g,d},$ as a compactification of $\Hdgf$. However,  $\Hdgf$ also admits the well-known compacification by \emph{admissible covers} constructed by Harris and Mumford~\cite{HarrisMumford} which we denote $\Hdgfb$. 

An admissible cover 
$$\pi:C\longrightarrow X$$ is a finite morphism of nodal curves of arithmetic genus $g(C)=g$ and $g(X)=0$ such that
\begin{enumerate}
\item
The smooth locus of $C$ is mapped to the smooth locus of $X$,
\item
The nodes of $C$ map to the nodes of $X$,
\item
At each node of $C$ the two branches have the same ramification order,
\item
The smooth locus of $X$ contains $2g+2d-3$ marked unordered branch points and one marked distinguished branch point,
\item
The target curve $X$ marked with the branch points is stable.
\end{enumerate}
The above also provide the admissible covers compactification $\Hdgb$ of $\Hdg$ where in this case (d) is replaced by the condition that the smooth locus of $X$ contains $2g+2d-2$ marked unordered simple branch points. Similarly, it provides the admissible covers compactification $\overline{\HH}_{g,d}^{o\text{ ad}}$ of ${\HH}_{g,d}^o$ the Hurwitz space of simply branched covers with ordered branch points. In this case (d) is replaced by the condition that the smooth locus of $X$ contains $2g+2d-2$ marked ordered simple branch points. This compactification comes with four natural forgetful morphisms.
$$f_{\text{br}}:\Hdgfb\longrightarrow \Mbar{0}{1+[2g+2d-3]}$$
is the dominant morphism that maps an admissible cover to the stabilisation of the target curve. 
$$f_{\text{ram}}:\Hdgfb\longrightarrow \Mbar{g}{1+[2g+2d-3]}$$
is the morphism that maps an admissible cover to the stabilisation of the source curve with the ramification points marked.
$$f_{\text{fib}}:\Hdgfb\longrightarrow \Mbar{g}{1+[d-2]}$$
is the morphism that maps an admissible cover to the stabilisation of the pointed source curve marked by the distinguished simply branched fibre. Finally,
$$\varphi:\Hdgfb\longrightarrow \Hdgb$$
is the natural forgetful morphism of generic degree $2g+2d-2$. Further, the first two forgetful morphisms naturally generalise to $\Hdgb$. 

Unfortunately, the admissible covers compactifications $\Hdgb$ and $\Hdgfb$ are not normal. However, the space of twisted stable maps~\cite{ACV} provide normalisations which we denote $\Hdgbv$ and $\Hdgfbv$ respectively and discuss the local structure of below. By an abuse of notation we denote by $f_{\text{br}},  f_{\text{ram}}, f_{\text{fib}}$ and $\tau$ the composition of the normalisation morphism and the morphisms defined above and similarly also let
$$\varphi:\Hdgfbv\longrightarrow\Hdgbv$$
denote the natural forgetful morphism. The meaning will be clear from context.

Now let $Y$ be the collection of boundary divisors of $\Hdgfbv$ with support contained in the pullback of the boundary of $\Mbar{g}{1+[d-2]}$ under $f_{\text{fib}}$. There is a natural morphism 
$$\tau:\Hdgfbv\setminus Y\longrightarrow\PHt$$ 
that respects the following commutative diagram.
\begin{displaymath}
\begin{tikzcd}
{\Hdgfbv\setminus Y}\ar[r,"\tau"]\ar[dr, "f_{\text{fib}}"']&\PHt\ar[d, "\pi"]\\
&\M{g}{1+[d-2]}&
\end{tikzcd}
\end{displaymath}
As discussed, $\tau$ is defined on a cover $\pi:C\longrightarrow \PP^1$ in the interior $\Hdgf$ by pulling back any differential with a unique double pole at the distinguished branch point and no other zeros or poles. This extends to the boundary as follows. Consider an admissible cover $\pi:\widetilde{C}\longrightarrow \widetilde{X}$ in $\Hdgfbv\setminus Y$. The stable model of the curve $\widetilde{C}$ and the points in the fibre of the marked branch point of $\pi$ is a smooth pointed curve $[C,p_1,p_2+\dots p_{d-1}]\in\M{g}{1+[d-2]}$ and is hence the result of the contraction of rational tails in $\widetilde{C}$.
   
Restricting $\pi$ to the curve $C$ we obtain $\pi|_C:C\longrightarrow X\cong\PP^1$ where due to the possible contraction of components in the source curve we obtain $\deg(\pi|_C)\leq d$ and the branching may be non-simple. Further, the specified branch point in $\widetilde{X}$ has a unique image $x$ in $X$ under the semi-stable reduction of the pointed base curve to $X$. Pulling back the unique differential (up to scaling by $\CC^*$) on $X$ with a double pole at $x$ and no other zeros or poles we obtain the required differential. 

Note that due to the stable reduction process, it may be the case that $\pi|_C(p_i)\ne x$ for some $p_i$. This corresponds to a twisted differential without a pole at $p_i$. This stable reduction process extends to families of admissible covers in $\Hdgfbv\setminus Y$ and there is a unique (up to scaling by $\CC^*$) non-vanishing family of differentials with a unique double pole at the marked branch point and no other zeros or poles on the target curves. Pulling back this family gives a family of non-vanishing differentials on the source curves and a unique family in $\PHt$.

\subsection{The boundary of $\PHt$}
The boundary $\PHt\setminus \Hdgf$ consists of four components corresponding to the four possibilities: two zeros of the differential colliding away from the $p_i$, two zeros of the differential are distinct and distinct from the $p_i$, but have the same image under the branch map obtained by integration, a zero of the differential collides with $p_1$, and two zeros of the differential collide with one of the points $p_2,\dots,p_d$. In this final case, the condition that the residues in the differential are zero necessitates that two simple zeros must collide with one of the double poles. A stable differential in which just one zero collided with one of the points $p_2,\dots,p_d$ would necessarily have a simple pole at that point and hence non-zero residue.

Let $T$ denote the closure of 
$${\PP\HH}_X(-3,\{-2^{d-2}\},2,{\{1^{2g+2d-5}\}})$$ 
in $\PHt$, the locus where two simple zeros collide.  

Let $D$ denote the closure of the locus where two distinct simple zeros have the same image under the morphism to a rational curve obtained by integrating the differential. This is the closure of the locus where the morphism has branching of type $(2,2,1^{d-4})$ at some point away from the distinguished branch point. 

Let $\delta_1$ denote the closure of 
$${\PP\HH}_X(-2,\{-2^{d-2}\},{\{1^{2g+2d-4}\}}),$$ 
the locus where one of the simple zeros in the differential collides with $p_1$. 

Finally, let $\delta_2$ denote the closure of 
$${\PP\HH}_X(-3,0,\{-2^{d-3}\},{\{1^{2g+2d-5}\}}),$$ 
the locus where two simple zeros in the differential collide with a point in $p_2+\dots+p_{d-2}$.

The irreducibility of all four divisors follows from Kluitmann~\cite{Kluitmann} and Natanzon~\cite{Natanzon} who proved that the Hurwitz spaces of covers of rational curves with arbitrary genus source curve and simple branching at all but one branch point are connected.

\subsection{The boundary of $\Hdgfb$ and $\Hdgfbv$}
The boundary of $\Hdgfb$ contains considerably more components. Consider admissible covers of a base containing two irreducible components and a unique node. We denote the ramification type of the cover $\pi:C\longrightarrow X$ above the unique node by $\underline{m}=[m_1,\dots,m_r]$ where $r$ is the number of nodes and $\sum m_i=\deg(\pi)=d$. The pullback under $f_{\text{br}}$ of each boundary divisor in $\Mbar{0}{1+[2g+2d-3]}$ will provide a number of divisors specified by signatures $\underline{m}$. These divisors will also in general be reducible. 

As discussed, $\Hdgfb$ is not normal. The space of twisted stable maps~\cite{ACV} which we denote by $\Hdgfbv$ provides a normalisation. Consider the space $\HH_{g,d}^o$ the Hurwitz space of ordered simple degree $d$ covers of rational curves by genus $g$ curves where all branch points are ordered. Denote by $\overline{\HH}_{g,d}^{o\text{ ad}}$ the admissible covers compactification and by $\overline{\HH}_{g,d}^{o}$ the normalisation, the space of twisted stable maps. We briefly recall the local structure near the boundary of $\overline{\HH}_{g,d}^{o}$ as it allows for a clean presentation. The boundary of $\Hdgfbv$ and $\Hdgbv$ are obtained by quotienting by $S_{2g+2d-3}$ and $S_{2g+2d-2}$ respectively. See~\cite{Ionel} for a more detailed description. 

Consider a general point in the boundary with branch profile above the unique node equal to $\underline{m}=[m_1,\dots,m_r]$. Let $q_i$ be the nodes in the nodal source curve corresponding to the entries $m_i$ sitting above the node $q$ in the rational nodal target curve. Locally near $q_i$ the source curve is given by $x_iy_i=s_i$ and near $q$ the target curve is given by $uv=t$ with $u=x_i^{m_i}$ and $v=y_i^{m_i}$. Introducing local parameters to normalise we obtain that $\overline{\HH}_{g,d}^{o}$ has 
$$\frac{\prod_{i=1}^r m_i}{\text{l.c.m}(\underline{m})}$$
branches along any irreducible component of the boundary of $\overline{\HH}_{g,d}^{o\text{ ad}}$ with signature $\underline{m}$, where l.c.m$(\underline{m})$ is the lowest common multiple of the entries in $\underline{m}$. The ramification along each branch in the branch morphism to $\Mbar{0}{2g+2d-2}$ is equal to $\text{l.c.m}(\underline{m})$.

We now restrict to specify a number of irreducible boundary divisors of $\Hdgfbv$ of interest to us. 

Let $\delta_{[m-1]}=\delta_{1+[n-m]}$ in $\Mbar{0}{1+[n-1]}$ be the divisor defined as the closure of the locus of stable pointed genus zero curves with a separating node with one component containing the specified point and $n-m$ other symmetrised points.

Let $\widetilde{T}_{1+[1]}$ and $\widetilde{T}_{[2]}$ be the divisors in the normalisation $\Hdgfbv$ corresponding to the unique irreducible components of ${f_{\text{br}}}^*\delta_{1+[1]}$ and ${f_{\text{br}}}^*\delta_{[2]}$ respectively, specified by branch profile above the unique node $\underline{m}=(3,1^{d-3})$.

Let $\widetilde{D}_{1+[1]}$ and $\widetilde{D}_{[2]}$ be the divisors in the normalisation $\Hdgfbv$ corresponding to unique irreducible components of ${f_{\text{br}}}^*\delta_{1+[1]}$ and ${f_{\text{br}}}^*\delta_{[2]}$ respectively, specified by branch profile above the unique node $\underline{m}=(2,2,1^{d-4})$.

Let $\widetilde{\Delta}_{1+[1]}$ and $\widetilde{\Delta}_{[2]}$ be the divisors in the normalisation $\Hdgfbv$ corresponding to the unique irreducible components of ${f_{\text{br}}}^*\delta_{1+[1]}$ and ${f_{\text{br}}}^*\delta_{[2]}$ respectively, specified by branch profile above the unique node $\underline{m}=(1^{d})$ and the condition that the source curve of a general element has an non-separating node.

Let $\widetilde{E}_{1+[2]}$ and $\widetilde{E}_{[3]}$ be the divisors in the normalisation $\Hdgfbv$ corresponding to the unique irreducible components of ${f_{\text{br}}}^*\delta_{1+[2]}$ and ${f_{\text{br}}}^*\delta_{[3]}$ respectively, specified by branch profile above the unique node $\underline{m}=(2,1^{d-2})$ and the condition that the source curve of a general element has two non-separating nodes with branch profile $(2,1)$.

Let $\widetilde{\delta}_{1+[1]}$ and $\widetilde{\delta}_{[2]}$ be the divisors in the normalisation $\Hdgfbv$ corresponding to the unique irreducible components of ${f_{\text{br}}}^*\delta_{1+[1]}$ and ${f_{\text{br}}}^*\delta_{[2]}$ respectively, specified by branch profile above the unique node $\underline{m}=(1^{d})$ and the condition that in general, the stable model of the source curve is a smooth genus $g$ curve.

Note that the divisors defined above are irreducible but not necessarily reduced. The irreducibility again follows from Kluitmann~\cite{Kluitmann} and Natanzon~\cite{Natanzon} who proved that the Hurwitz spaces of covers of rational curves with simple branching at all but one branch point are connected. A general admissible cover corresponding to each divisor is depicted in Figure~\ref{Divisors}. 

\begin{figure}[htbp]
\begin{center}
\begin{overpic}[width=0.85\textwidth]{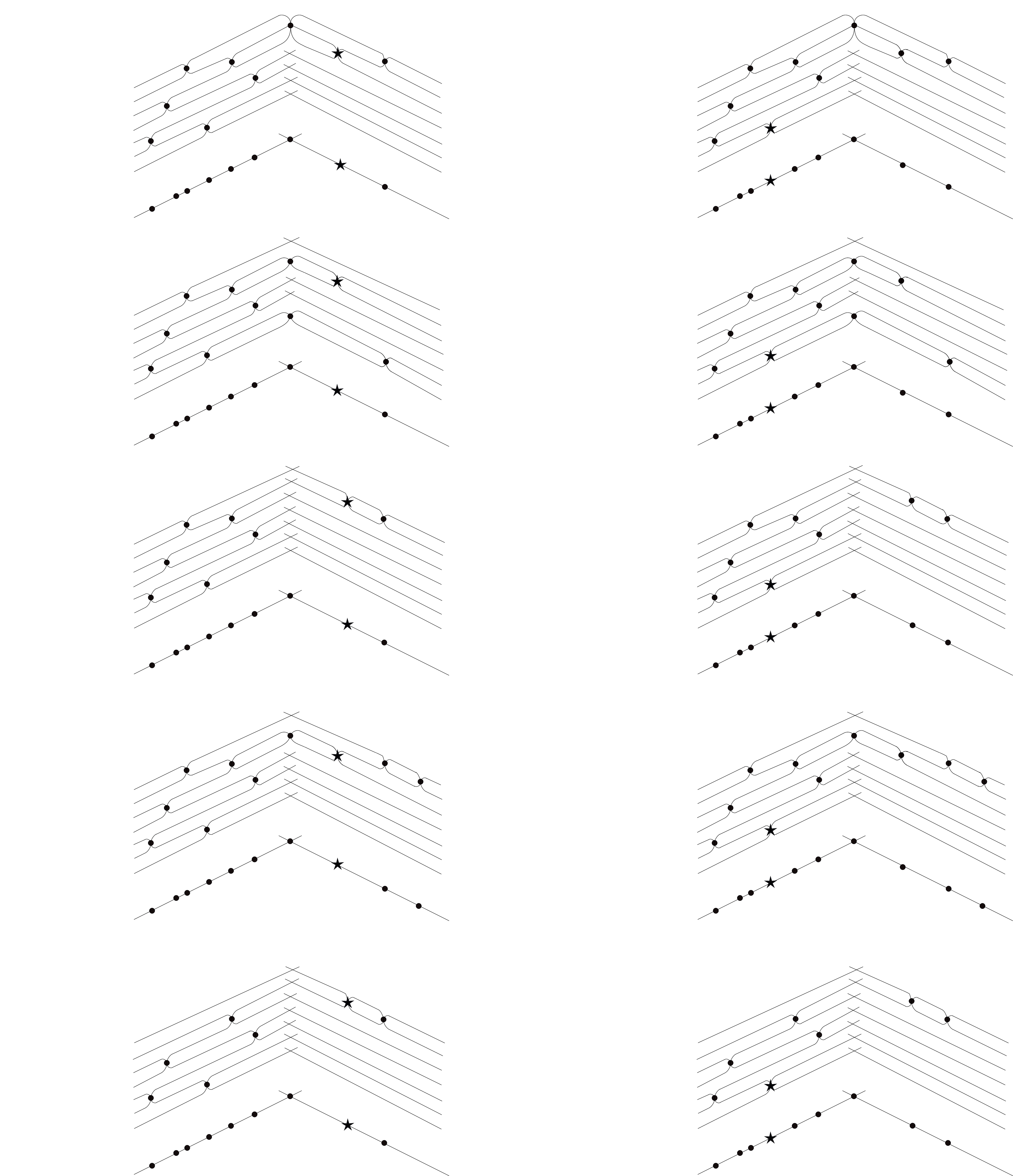}

\put(3,88){$\widetilde{T}_{1+[1]}$}
\put(53,88){$\widetilde{T}_{[2]}$}

\put(3,68){$\widetilde{D}_{1+[1]}$}
\put(53,68){$\widetilde{D}_{[2]}$}

\put(3,48){$\widetilde{\Delta}_{1+[1]}$}
\put(53,48){$\widetilde{\Delta}_{[2]}$}

\put(3,28){$\widetilde{E}_{1+[2]}$}
\put(53,28){$\widetilde{E}_{[3]}$}

\put(3,8){$\widetilde{\delta}_{1+[1]}$}
\put(53,8){$\widetilde{\delta}_{[2]}$}
\end{overpic}
 \caption{Boundary divisors of interest in $\Hdgfbv$ }
  \label{Divisors}
\end{center}
\end{figure}

\subsection{Test curves in $\Hdgfb$ and $\Hdgfbv$}
To prove the linear independence of boundary components we construct a number of test curves. Though some constructions take place in the space of admissible covers we consider the appropriately lift in the normalisation $\Hdgfbv$ scaled by the ramification order of the unique irreducible boundary component the curve lies inside.

For any general points $p_i$ on a general genus $g$ curve $C$ by Riemann-Roch we have for $d>g+1$,
$$\dim H^0(2p_1+p_2+\dots +p_{d-1})=1-g+d\geq 3.$$
Choosing a general linear series 
$$|V|\subset \PP H^0(2p_1+p_2+\dots +p_{d-1})$$
of dimension one not containing $2p_1+p_2+\dots +p_{d-1}$ we can hence construct a curve in $\Hdgfbv$ by constructing the unique cover specified by $2p_1+p_2+\dots +p_{d-1}$ and each element in  $|V|$. Let $F$ denote this curve. When $d=g+1$ we can construct this curve by requiring that the $p_i$ satisfy
$$\dim H^0(\omega_C(-2p_1-p_2-\dots -p_{d-1}))=1.$$

\begin{prop}
For $d\geq g+1$, the curve $F$ satisfies the following intersection numbers
$$F\cdot \widetilde{T}_{[2]}=3(2g+d-2),\hspace{0.3cm}F\cdot\widetilde{D}_{[2]}=2(g^2+2gd+d^2-5d-7g+6),\hspace{0.3cm}F\cdot\widetilde{\delta}_{1+[1]}=1,\hspace{0.3cm}F\cdot\widetilde{\delta}_{[2]}=d-2$$
and the intersection with every other boundary divisor is zero.
\end{prop}

\begin{proof}
The intersection $F\cdot \widetilde{T}_{[2]}$ is obtained as the number of ramification points in a general linear series of projective dimension $r=2$ degree $d$ given by the Pl\"{u}cker formula as
$$(r+1)d+(r+1)r(g-1).$$
The intersection $F\cdot\widetilde{D}_{[2]}$ is obtained by the generalisation of the Pl\"{u}cker formula known as de Jonqui\`{e}res' formula~\cite[pg. 359]{ACGH}. The number of section of the type $(2,2,1^{d-4})$ in a general linear series of projective dimension $r=2$ degree $d$ is given by the coefficient of $x^2y^{d-4}$ in
$$(1+4x+y)^g(1+2x+y)^{d-r-g}.$$
The final two intersections are clear as requiring that a ramification point collide with any $p_i$ specifies a unique section in the linear system.
\end{proof}

Now fix a general degree $d$ cover $\nu:C\longrightarrow \PP^1$ by a genus $g-1$ curve $C$ with only simple branching and a simply branched trigonal cover $\alpha:X\longrightarrow \PP^1$ of the rational line by a rational curve $X$. We obtain an admissible cover in $\Hdgbv$ by gluing together a ramification point from each cover and identifying the other point in the fibre of $\alpha$ to one of the points in the same fibre of $\nu$. By varying one of the remaining branch points in $\alpha$ and fixing the other three we obtain a curve in $\Hdgbv$. By marking a distinct ramification point in $\nu$ or the ramification point in $\alpha$ that is being moved we obtain the curves $G_{[3]}$ and $G_{1+[2]}$ respectively. This situation is depicted in Figure~\ref{Gcurves}.

\begin{figure}[htbp]
\begin{center}
\begin{overpic}[width=0.80\textwidth]{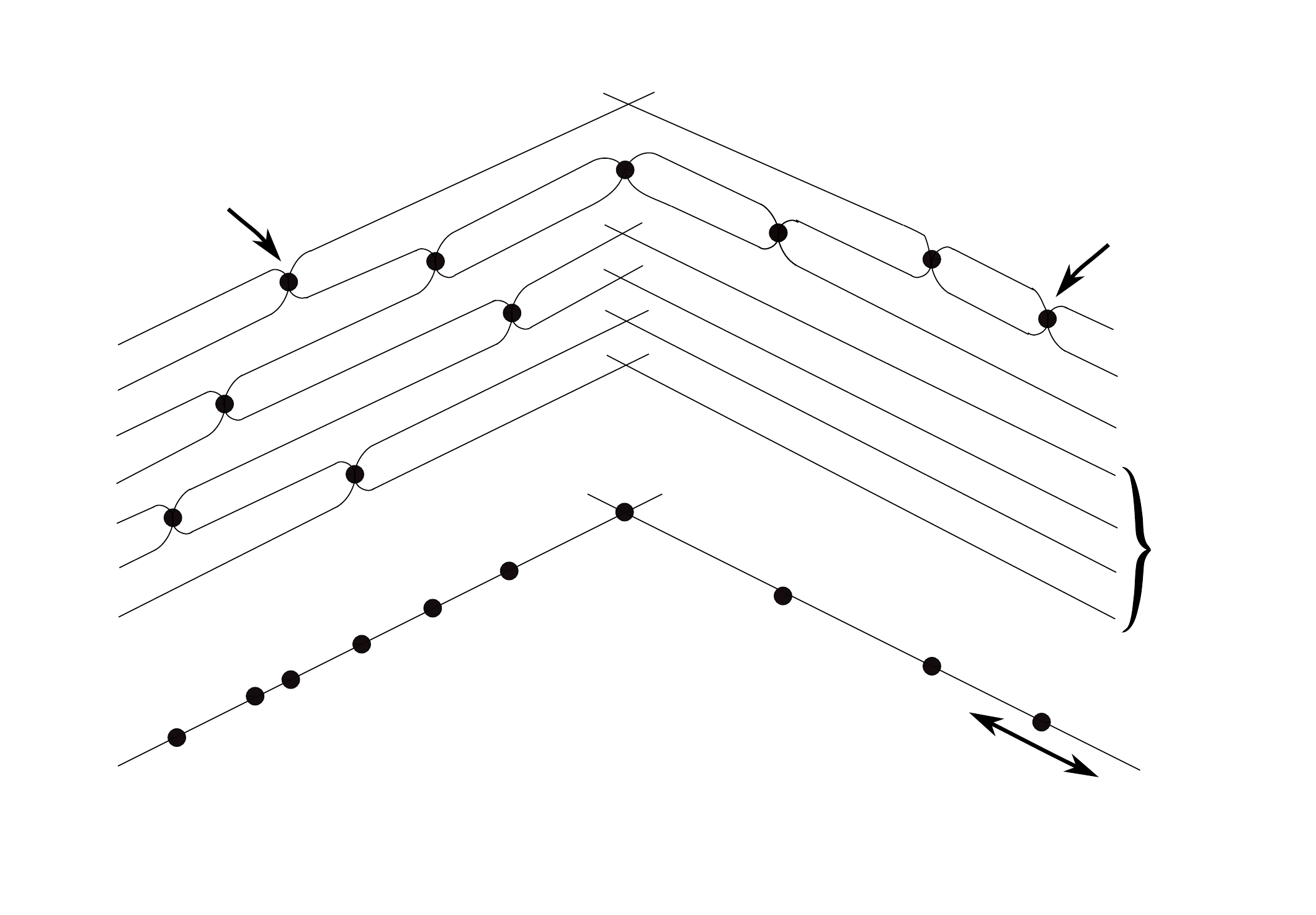}
\put(24,12){$\nu:C\longrightarrow \PP^1$ }
\put(60,58){$\alpha:X\longrightarrow \PP^1$}

\small{\put(85,53){Mark this ramification}
\put(85,50){ point to obtain $G_{1+[2]}$ }}

\small{\put(0,58){Mark this ramification}
\put(0,55){ point to obtain $G_{[3]}$ }}

\small{\put(66,7){Allow this branch point}
\put(69,4){ to move freely }}
\small{\put(87,29){ Unramified}
\put(87,26){ rational tails }}
\end{overpic}
 \caption{Constructing the test curves $G_{[3]}$ and $G_{1+[2]}$. }
  \label{Gcurves}
\end{center}
\end{figure}

\begin{prop}
The curves $G_{[3]}$ and $G_{1+[2]}$ satisfy the following intersection numbers
$$G_{[3]} \cdot\widetilde{T}_{[2]}=9,\hspace{0.7cm}G_{[3]} \cdot\widetilde{\Delta}_{[2]} =3,\hspace{0.7cm}G_{[3]} \cdot\widetilde{E}_{[3]}=-4,  $$
and
$$G_{1+[2]} \cdot\widetilde{T}_{1+[1]}=6,\hspace{0.7cm}G_{1+[2]} \cdot\widetilde{T}_{[2]}=3,\hspace{0.7cm}G_{1+[2]} \cdot\widetilde{\Delta}_{1+[1]} =2,\hspace{0.7cm}G_{1+[2]} \cdot\widetilde{\Delta}_{[2]} =1,\hspace{0.7cm}G_{1+[2]} \cdot\widetilde{E}_{1+[2]}=-4.  $$
The intersection with every other boundary divisor is zero.
\end{prop}

\begin{proof}
Fix 
$$\alpha: X\longrightarrow \PP^1$$
a degree three cover of the rational line by a rational curve simply branched above $b_1,b_2,b_3$ and $b_4$. Removing these points and their pre-images under $\alpha$ we obtain an  \'{e}tale cover. The monodromy representation
$$\rho:\pi_1(\PP^1\setminus\{b_1,b_2,b_3,b_4\})\longrightarrow S_3$$
describes how cycles in $\pi_1(\PP^1\setminus\{b_1,b_2,b_3,b_4\})$ permute the sheets of the \'{e}tale cover 
$$X\setminus\{\alpha^{-1}(b_i)\}_{i=1}^4\longrightarrow \PP^1\setminus\{b_1,b_2,b_3,b_4\}.$$
For the fixed cover $\alpha$, this representation is unique up to simultaneous conjugation (labelling of the sheets). Further, when the image of any representation of $\pi_1(\PP^1\setminus\{b_1,b_2,b_3,b_4\})$ in $S_3$ is a transitive subgroup it gives a unique isomorphism class of an irreducible cover by the Riemann existence theorem.

Let $\gamma_i$ for $i=1,2,3,4$ be the basis for $\pi_1(\PP^1\setminus\{b_1,b_2,b_3,b_4\})$ obtained by taking small cycles enclosing each branch point such that
$$\gamma_1\cdot\gamma_2\cdot\gamma_3\cdot\gamma_4=\text{Id}.$$ 
Then $\rho(\gamma_i)=\tilde{\gamma}_i$ for $i=1,2,3,4$ is a transposition in $\mathfrak{S}_3$. To enumerate the number of such covers we enumerate the number of such monodromy representations up to conjugation. That is, enumerate the choices of transpositions $\tilde{\gamma}_i$ up to conjugation such that 
$$\tilde{\gamma}_1\cdot\tilde{\gamma}_2\cdot\tilde{\gamma}_3\cdot\tilde{\gamma}_4=\text{Id}.$$ 

Fixing $\tilde{\gamma}_1=(1, 2)$ we observe choosing $\tilde{\gamma}_2$ and $\tilde{\gamma}_3$ uniquely specifies $\tilde{\gamma}_4$ which is necessarily an odd permutation and hence in the case of $\mathfrak{S}_3$, a transposition. Eliminating the possibility that all $\tilde{\gamma}_i=(1, 2)$ and conjugating by $(1, 2)$ we obtain the number of such covers to be
$$\frac{3^2-1}{2}=4.$$
A representative of each conjugacy class is given below
\begin{center}
\vspace{0.3cm}
\begin{tabular}{|p{2.8cm} |p{1.5cm}|p{1.5cm}|p{1.5cm}|p{1.5cm}|  }
 \hline
 Conjugacy Class&$\tilde{\gamma}_1$ & $\tilde{\gamma}_2$ & $\tilde{\gamma}_3$&$\tilde{\gamma}_4$  \\
 \hline\hline
 \hspace{1.1cm}A& $(1, 2)$ & $(1, 2)$ & $(1, 3)$&$(1, 3)$  \\
 \hline
 \hspace{1.1cm}B & $(1, 2)$ & $(1, 3)$ &$(1, 3)$& $(1, 2)$  \\
 \hline
 \hspace{1.1cm}C &$(1, 2)$ & $(1, 3)$ &$(2, 3)$ & $(1, 3)$ \\
 \hline
 \hspace{1.1cm}D& $(1, 2)$ & $(1, 3)$ &$(1, 2)$& $(2, 3)$  \\
 \hline
\end{tabular}
\vspace{0.3cm}
\end{center}
All intersection numbers but the last entry in each row of the proposition follow from this table which enumerates the way ramification points collide when branch points collide.

The normal bundles of $\widetilde{E}_{[3]}$ and $\widetilde{E}_{1+[2]}$ can be restricted to $G_{[3]}$ and $G_{1+[2]}$
respectively to obtain the final entry. We give an alternate argument considering the intersection of the pushforward of these curves to $\Mbar{0}{1+[2g+2d-3]}$ and applying the projection formula to the finite morphism ${f_{\text{br}}}$. We have
$$({f_{\text{br}}}_*G_{[3]})\cdot\delta_{[2]}=12\hspace{0.3cm}\text{and}\hspace{0.3cm}({f_{\text{br}}}_*G_{[3]})\cdot\delta_{[3]}=-4,$$
while
$$({f_{\text{br}}}_*G_{1+[2]})\cdot\delta_{[2]}=4,\hspace{0.3cm}({f_{\text{br}}}_*G_{1+[2]})\cdot\delta_{1+[1]}=8,\hspace{0.3cm}({f_{\text{br}}}_*G_{1+[2]})\cdot\delta_{[3]}=0,\hspace{0.3cm}\text{and}\hspace{0.3cm}({f_{\text{br}}}_*G_{1+[2]})\cdot\delta_{1+[2]}=-4.$$
Now observe that $\widetilde{E}_{[3]}$ is the only component of ${f_{\text{br}}}^*\delta_{[3]}$ that the curve $G_{[3]}$ intersects. Similarly, $\widetilde{E}_{1+[2]}$ is the only component of ${f_{\text{br}}}^*\delta_{1+[2]}$ that the curve $G_{1+[2]}$ intersects. Hence the two negative intersection numbers follow by an application of the projection formula.
\end{proof}

Now define the test curve $A_h$ for $h=0,\dots,g-1$ in $\Hdgfbv$ by decorating curves used by Deopurkar and Patel~\cite{DP} with marked ramification points, branch points and one distinguished branch point.

Let 
$$\alpha_b:X_b\longrightarrow \PP^1$$ 
be a family of hyperelliptic curves of genus $g-h-1$ obtained by taking a double cover $X\longrightarrow \PP^1\times\PP^1$ branched along a curve of bi-degree $(2(g-h-1)+2,2).$ To obtain sections $p:B\longrightarrow X$ and $q:B\longrightarrow X$ over $\{0\}\times\PP^1$ and $\{\infty\}\times \PP^1$, let the branch divisor be tangent to $\{0\}\times\PP^1$ and $\{\infty\}\times \PP^1$.

Let 
$$\gamma:Y\longrightarrow \PP^1$$ 
be a generic degree $d-2$ cover by a smooth genus $h$ curve $Y$. Let $y_0$ and $y_\infty$ be fixed points in the unramified fibres $\gamma^{-1}(0)$ and $\gamma^{-1}(\infty)$.

Let 
$$\beta_i:Z_i\longrightarrow \PP^1$$ 
for $i=1,2$ be fixed double covers for rational curves $Z_i$ and let $t_i$ and $\overline{t}_i$ be fixed distinct conjugate points under $\beta_i$ for $i=1,2$.

To obtain the curve $A_h$ in $\Hdgfbv$ identify the sections $p$ with $t_1$ and $q$ with $t_2$ and points $y_0$ with $\overline{t}_1$ and $y_\infty$ with $\overline{t}_2$. Mark the fibre above the image of one of the ramification points of $\gamma:Y\longrightarrow \PP^1$. This construction is depicted in Figure~\ref{AhBh}.

\begin{figure}[htbp]
\begin{center}
\begin{overpic}[width=1.03\textwidth]{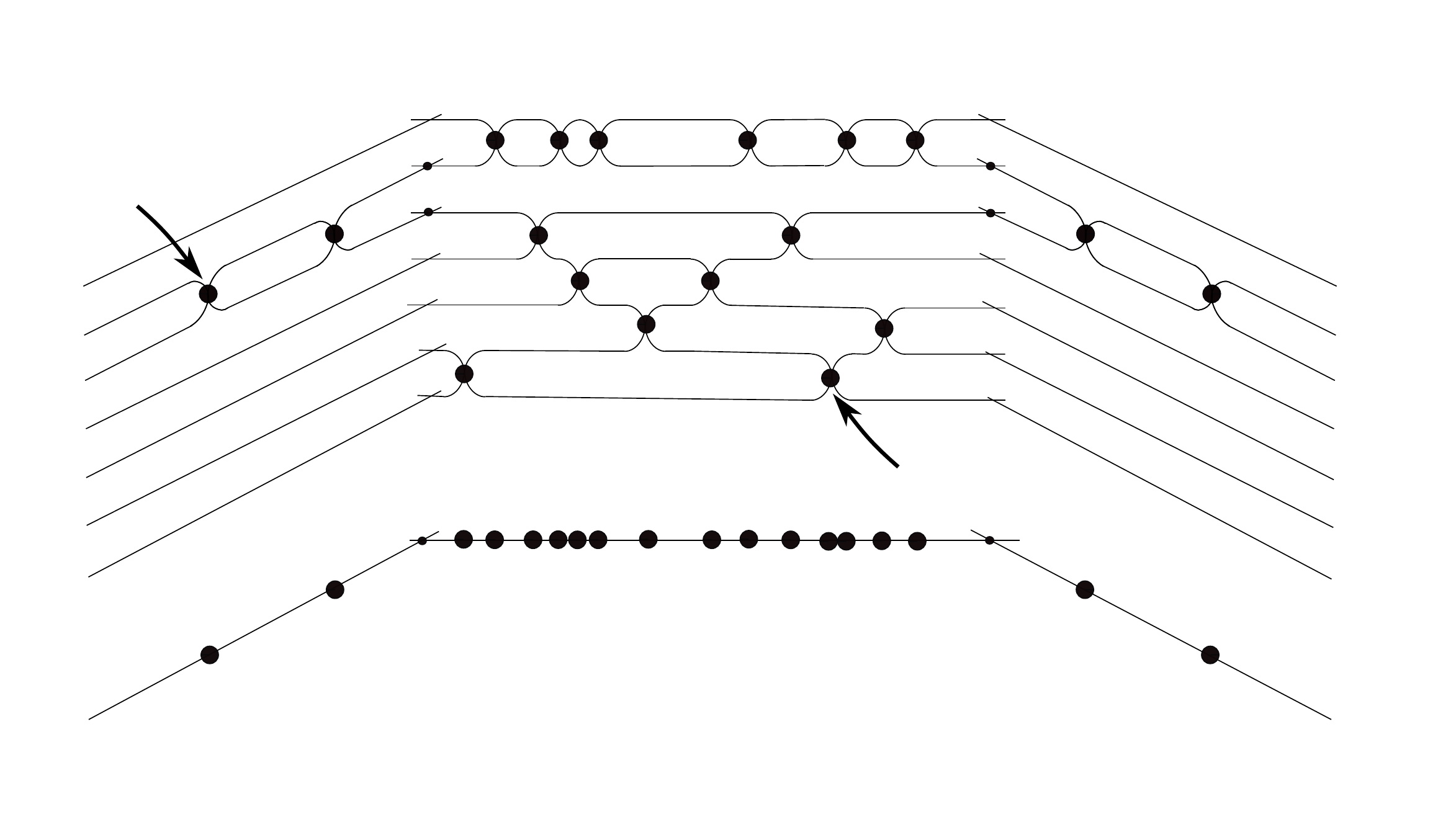}

\put(14,9){$\beta_1:Z_1\longrightarrow \PP^1$ }
\small{\put(10,5){and $d-2$ unramified}
\put(13,2){ rational tails }}

\put(71,9){$\beta_2:Z_2\longrightarrow \PP^1$ }
\small{\put(67,5){and $d-2$ unramified}
\put(70,2){ rational tails }}

\put(40,54){$\alpha_b:X_b\longrightarrow \PP^1$}

\put(40,27){$\gamma:Y\longrightarrow \PP^1$}

\small{\put(63,27){Mark this ramification}
\put(63,24){ point to obtain $A_h$ }}

\small{\put(0,49){Mark this ramification}
\put(0,46){ point to obtain $B_h$ }}

\small{\put(27,47){$t_1$}
\put(27,43.5){$\overline{t}_1$}
\put(30,48){$p$}
\put(30,42){$y_0$}

\put(70.5,47){$t_2$}
\put(70.5,43.5){$\overline{t}_2$}
\put(68.5,48){$q$}
\put(68.5,42){$y_\infty$}

\put(68,18){$\infty$}
\put(29,18){$0$}
}

\end{overpic}
 \caption{Constructing the test curves $A_h$ and $B_h$. }
  \label{AhBh}
\end{center}
\end{figure}

\begin{prop}\label{prop:Ah}
For $h=0,\dots,g-1$, the curve $A_h$ satisfies the following intersection numbers
$$A_h\cdot\widetilde{D}_{1+[1]}=2,\hspace{0.7cm}A_h\cdot\widetilde{D}_{[2]}=4(d+h)-14,\hspace{0.7cm}A_h\cdot\widetilde{E}_{[3]}=2,\hspace{0.7cm}A_h\cdot\widetilde{\Delta}_{[2]}=8(g-h)-8,     $$
and the intersection with every other boundary divisor is zero.
\end{prop}

\begin{proof}
Denote by $S$ the branching curve in $\PP^1\times \PP^1$ of bi-degree $(2(g-h),2)$. Let $\alpha$ and $\beta$ be the classes of the fibre of the projection onto the first and second component of $\PP^1\times\PP^1$ respectively. Hence $S$ has class $2(g-h)\alpha+2\beta$ and by adjunction, the genus of $S$ is $2(g-h)-1$. Riemann-Hurwitz or the Pl\"{u}cker formula gives that the projection of this curve onto the second factor will have $8(g-h)-4$ simple ramification points. 

The intersections with $\widetilde{D}_{1+[1]}$ occur for $b$ such that a ramification point of $\alpha_b$ sits above the marked ramification point of $\gamma$ which is equal to $\alpha\cdot [S]=2$.

Similarly, the intersections with $\widetilde{D}_{[2]}$ occur for $b$ such that a ramification point of $\alpha_b$ sits above the remaining ramification points of $\gamma$. Hence 
$$A_h\cdot\widetilde{D}_{[2]}=(\alpha\cdot [S])(2h+2(d-2)-3).$$

The intersection with $\widetilde{E}_{[3]}$ is computed as 
$$A_h\cdot \widetilde{E}_{[3]}=([p]+[q])\cdot \ram(\alpha)=2$$

Finally, the intersections with $\widetilde{\Delta}_{[2]}$ is equal to the restriction of the normal bundle of $\widetilde{\Delta}_{[2]}$ to the curve $A_h$. This has two contributions. First, there are $8(g-h)-4$  intersections from the family $\alpha$. The second contribution is at the nodes $[p]$ and $[q]$ that contribute 
$$[p]^2+[q]^2=-4$$
to the intersection.

Again, all these intersection numbers can be checked via pushing forward the curve under the branch morphism to $\Mbar{0}{1+[2g+2d-3]}$.
\end{proof}

The curve $B_h$ for $h=0,\dots,g-1$ in $\Hdgfbv$ is obtained by following the same construction as $A_h$, but distinguishing a different ramification point. In this case we distinguish a ramification point of $\beta_1:Z_1\longrightarrow \PP^1$. This is depicted in Figure~\ref{AhBh}.

\begin{prop}
For $h=0,\dots,g-1$, the curve $B_h$ satisfies the following intersection numbers
$$B_h\cdot\widetilde{D}_{[2]}=4(d+h)-12,\hspace{0.7cm}B_h\cdot\widetilde{E}_{1+[2]}=1,\hspace{0.7cm}B_h\cdot\widetilde{E}_{[3]}=1,\hspace{0.7cm}B_h\cdot\widetilde{\Delta}_{1+[1]}=-2,
\hspace{0.7cm}B_h\cdot\widetilde{\Delta}_{[2]}=8(g-h)-6,   $$
and the intersection with every other boundary divisor is zero.
\end{prop}

\begin{proof}
The proof follows clearly from the proof of Proposition~\ref{prop:Ah} by simply distinguishing a different branch point.
\end{proof}

\subsection{The Picard rank conjecture}
Patel~\cite{Patel} showed that the irreducible components of the boundary of $\Hdgbv$ are linearly independent. Though this is not known to hold on $\Hdgfbv$, the following weaker statement on linear independence will suffice for our purposes. 

\begin{prop}\label{prop:zero}
For $d>g$, if the relation
$$\sum c(\delta)\delta=0$$
holds in $A^1_\QQ(\Hdgfbv)$ on the irreducible boundary components of $\Hdgfbv$ then
$$c(\widetilde{T}_{1+[1]})=c(\widetilde{T}_{[2]})=c(\widetilde{D}_{1+[1]})=c(\widetilde{D}_{[2]})= c(\widetilde{\delta}_{1+[1]})=c(\widetilde{\delta}_{[2]})=c(\widetilde{\Delta}_{1+[1]})=c(\widetilde{\Delta}_{[2]})=c(\widetilde{E}_{1+[2]})=c(\widetilde{E}_{[3]})=0.$$
\end{prop}

\begin{proof}
We prove the existence of the following non-singular matrix relating the coefficients of the above divisors in any such relation
$$M=\begin{pmatrix}
2&b-2&0&0&0&0&0&0&0&0\\
0&0&2&b-2&0&0&0&0&0&0\\
0&0&0&0&2&b-2&0&0&0&0\\
0&0&0&0&0&0&2&b-2&0&0\\
0&0&0&0&0&0&0&0&3&b-3\\
0&0&2&4 d - 14&0&0&0&8 g - 8&0&2\\
0&0&0&4 d - 12&0&0&-2&8 g - 6&1&1\\
0&0&2&4 d - 10&0&0&0&8 g - 16&0&2\\
0&9&0&0&0&0&0&3&0&-4\\
0&3 (2 g + d - 2)&0&2 (g^2 + 2 g d + d^2 - 5 d - 7 g + 6)&1&d - 2&0&0&0&0
\end{pmatrix}$$
where $b=2g+2d-2$ and observe
$$\det(M)=-2304 g (g+d-1) (2g+2d-5)\ne 0.$$

Patel~\cite{Patel} showed that the irreducible components of the boundary of $\Hdgbv$ are linearly independent. Consider the finite morphism
$$\varphi:\Hdgfbv\longrightarrow \Hdgbv$$
that forgets the specified branch point. 
Pushing forward the relation under $\varphi$ we obtain the first five rows of $M$ from the coefficients of irreducible boundary divisors in a relation in $A^1_\QQ(\Hdgbv)$ that must be trivial. 
The remaining four rows are obtained via intersection of the relation with the curves $A_0,B_0, A_1, G_{[3]}$ and $F$. 
\end{proof}

This allows us to show that in the case $d>g$ the boundary components of $\PHt$ are linearly independent.

\begin{prop}\label{prop:bdrylinindep}
For $d>g$ and $g\geq 3$, the components $T$, $D$, $\delta_1$ and $\delta_2$ of $\PHt\setminus\Hdgf$ are linearly independent in $A^1_\QQ(\PHt)$. 
\end{prop}

\begin{proof}
Recall the birational morphism 
$$\tau:\Hdgfbv\setminus Y\longrightarrow\PHt$$ 
where $Y$ is the pullback of the boundary of $\Mbar{g}{1+[d-2]}$ under $f_{\text{fib}}$. 
The divisors $\widetilde{T}, \widetilde{D}_{[2]}, \widetilde{\delta}_{[1]}$ and $\widetilde{\delta}_{[2]}$ have image $T$, $D$, $\delta_1$ and $\delta_2$ respectively in $\PHt$. 

Any non-trivial relation in $T$, $D$, $\delta_1$ and $\delta_2$ in $A^1_\QQ(\PHt)$ pulls back under ${\tau}$ to a relation in the boundary components of $\Hdgfbv\setminus Y$ with a non-zero coefficients for at least one of $\widetilde{T}, \widetilde{D}_{[2]}, \widetilde{\delta}_{[1]}$ and $\widetilde{\delta}_{[2]}$ in $A^1_{\QQ}(\Hdgfbv\setminus Y)$. Proposition~\ref{prop:zero} then provides the contradiction.
\end{proof}

This provides the first proposition that contributes to Theorem~\ref{thm:PicRank}.
\begin{prop}
$\Pic_{\QQ}(\Hdgf)$ and $\Pic_{\QQ}(\Hdg)$ for $d>g$ are trivial.
\end{prop}
\begin{proof}
For $d>g$, the rank of the rational Chow group of $\PHt$ given by Corollary~\ref{cor:ChowRankd>g+1} and Corollary~\ref{cor:ChowRankd=g+1} is completely accounted for by the linearly independent boundary components given by Proposition~\ref{prop:bdrylinindep}. Hence $A^1_\QQ(\Hdgf)=0$ and as $\Hdgf$ is normal, $\Pic_{\QQ}(\Hdgf)\longrightarrow A^1_\QQ(\Hdgf)$ is injective and hence $\Pic_{\QQ}(\Hdgf)$ is trivial.

The forgetful morphism 
$$\varphi: \Hdgf\longrightarrow \Hdg$$
is finite and proper and hence 
$$ \varphi_*\circ\varphi^*:A^1_\QQ(\Hdg)\longrightarrow A^1_\QQ(\Hdg) $$ 
is equal to multiplication by $(2g+2d-2)$. However, $A^1_\QQ(\Hdgf)=0$ and hence $\varphi^*[Z]=0$ for any $[Z]\in A^1_{\QQ}(\Hdg)$. Hence
$$0=\varphi_*\circ\varphi^*[Z]=(2g+2d-2)[Z]$$
and $A^1_{\QQ}(\Hdg)=0$ which implies $\Pic_{\QQ}(\Hdg)=0$. 
\end{proof}

In the remainder of this section we prove the one remaining case of Theorem~\ref{thm:PicRank} that $\Pic_\QQ(\Hdg)$ is trivial for $d=g$. In this case the map
$$\PHt\longrightarrow \Mbar{g}{1+[d-2]}$$
is no longer dominant and we must find a new strategy to bound the rational Chow rank. For any smooth genus $g$, degree $g$ cover of the rational line
$$\pi:C\longrightarrow \PP^1$$
the line bundle $\omega_C\otimes\pi^*\OO(-1)$ has degree $g-2$ and by Riemann-Roch the dimension is equal to
$$\dim H^0(C,\omega_C\otimes\pi^*\OO(-1))=\dim H^0(C,\pi^*\OO(1))-1.$$
Let $\mathcal{C}_{g,g-2}$ be the universal degree $g-2$ symmetric product, the moduli spaces of tuples of $g-2$ unordered possibly non-distinct points on smooth genus $g$ curves. We obtain a birational map
\begin{equation}\label{eq:HHgg}
\begin{array}{rcl}
\HH_{g,g}&\dashedrightarrow&\mathcal{C}_{g,g-2}\\
\left[\pi:C\to\PP^1\right]&\mapsto&\left[C,\omega_C\otimes\pi^*\OO(-1)\right]\\ 
\end{array}
\end{equation}
that is well-defined when
$$\dim H^0(C,\pi^*\OO(1))=2$$
and has a well defined inverse at $[C,p_1+\dots+p_{g-2}]$ in $\mathcal{C}_{g,g-2}$ when
$$\dim H^0(C,\OO_C(p_1+\dots+p_{g-2}))=1$$ 
and the linear system $|A|$ for $A=\omega_C(-p_1-\dots-p_{g-2})$ has no base points and the map $|A|$ has simple branching. 

\begin{figure}[htbp]
\begin{center}
\begin{overpic}[width=1.0\textwidth]{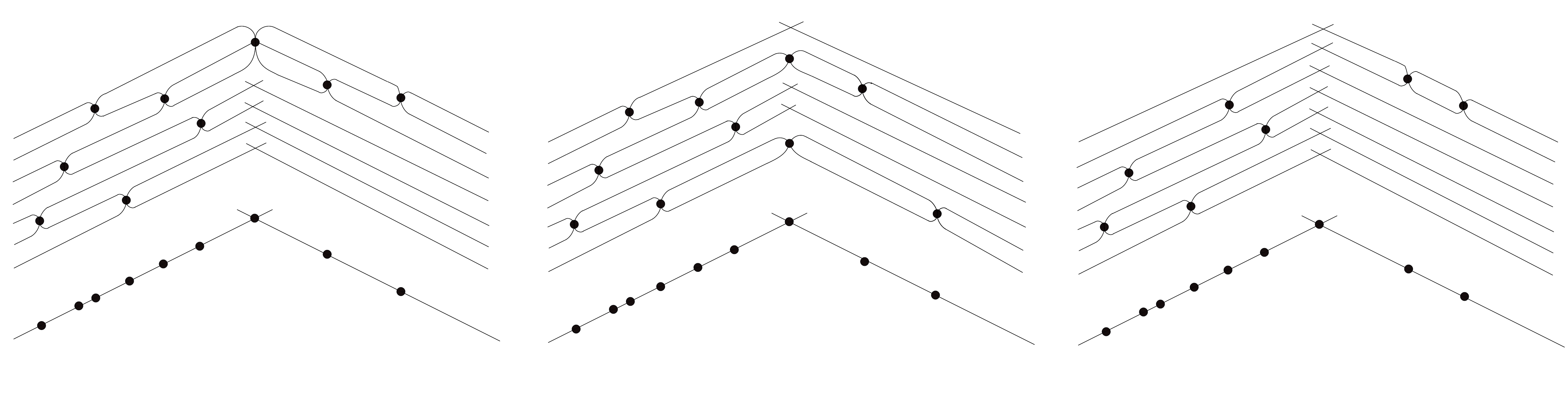}
\put(84,3){$\delta$}
\put(15,3){$T$}
\put(50,3){$D$}
\end{overpic}
 \caption{Boundary divisors of interest in $\Hdgb$ }
  \label{Divisors2}
\end{center}
\end{figure}

Let $Y$ be the collection of all boundary divisors in $\overline{\HH}_{g,g}$ excluding only the three irreducible components of $f_{\text{br}}^*\delta_{[2]}$ such that the stable model of the general source curve is smooth. We denote these divisors by $T, {D}$ and $\delta$ and they are specified as the irreducible components of $f_{\text{br}}^*\delta_{[2]}$ corresponding to $\underline{m}=[3,1^{d-3}],[2,2,1^{d-4}]$ and $[1^d]$ respectively. A general cover in each divisor is depicted in Figure~\ref{Divisors2}. Then $\overline{\HH}_{g,g}\setminus Y$ provides a partial compactification of $\HH_{g,g}$ by admissible covers with the conditions that the stable model of the source curve is smooth and at most two branch points come together as under these conditions the space of admissible covers and twisted stable maps coincide. Further, the map $(\ref{eq:HHgg})$ extends to a birational morphism
$$\begin{array}{rcl}
\overline{\HH}_{g,g}\setminus Y&\dashedrightarrow&\mathcal{C}_{g,g-2}\\
\left[\pi:C\to\PP^1\right]&\mapsto&\left[C,\omega_C\otimes\pi^*\OO(-1)\right]\\ 
\end{array}$$
that is well-defined when
$$\dim H^0(C,\pi^*\OO(1))=2$$
hence 
$$\dim H^0(C,\omega_C\otimes\pi^*\OO(-1)=1$$
and has a well defined inverse at $[C,p_1+\dots+p_{g-2}]$ in $\mathcal{C}_{g,g-2}$ when
$$\dim H^0(C,\OO_C(p_1+\dots+p_{g-2}))=1$$ 
and the resulting cover from the linear system $|\omega_C(-p_1-\dots-p_{g-2})|$ has branching that is the result of at most two branch points coming together which we detail explicitly below. 

Let $A=\omega_C(-p_1-\dots-p_{g-2})$ and $\dim H^0(C,A)=2$. If $|A|$ is base point free and the branching of the map $|A|$ is simple or of type $(3,1^{d-3})$ or $(2,2,1^{d-4})$ then $[C,p_1+\dots+p_{g-2}]$ gives a well-defined inverse. If $|A|$ is not base point free and the base points of $|A|$ are given by $q_1,\dots,q_m$ then to obtain a well-defined inverse we require 
\begin{enumerate}
\item
the base points are simple,
\item
the branching of $|A(-q_1-\dots-q_m)|$ is simple, of type $(3,1^{d-m-3})$ or $(2,2,1^{d-m-4})$,
\item
the $q_i$ sit in distinct fibres of $|A(-q_1-\dots-q_m)|$ which are unbranched.
\end{enumerate}
In both cases the requirements ensure that the admissible cover can be obtained by attaching rational tails in a unique way specified by Figure~\ref{Divisors2}. The unique way to attach rational tails at ramification of type $(3,1^{d-3})$ or $(3,1^{d-m-3})$ is depicted in $T$, ramification of type $(2,2,1^{d-4})$ or $(2,2,1^{d-m-4})$ is depicted in $D$ and rational tails are attached at simple base points in unramified fibres in the unique way specified in $\delta$. This construction clearly extends to families.

Let $\widetilde{X}$ be the closure in $\overline{\HH}_{g,g}$ of the locus of covers $\{\pi:C\longrightarrow \PP^1\}$ in ${\HH}_{g,g}$ such that 
$$\dim H^0(C,\pi^*\OO(1))\geq3,$$
let $X_1$ be the locus of $[C,p_1+\dots+p_{g-2}]$ in $\mathcal{C}_{g,g-2}$ such that
$$\dim H^0(C,\OO_C(p_1+\dots+p_{g-2}))\geq2$$
and let $X_2$ be the locus where 
$$\dim H^0(C,\OO_C(p_1+\dots+p_{g-2}))=1$$
but $A=\omega_C(-p_1-\dots-p_{g-2})$ does not satisfy the requirements above.  We have 
$$\codim(\widetilde{X})=\codim(X_1)=\codim(X_2)=2.$$
The first two codimensions follow from a simple dimension count, while the third is because $X_2$ is the image of the divisorial locus where more than two branch points come together and the stable model of the source curve is smooth. This locus is contracted as the information of the position of the colliding branch points is forgotten.

This results in the following relation on the rank of rational Chow groups.
\begin{prop}
The following relation on the rank of the rational Chow groups holds
$$\rho(\overline{\HH}_{g,g}\setminus Y)=\rho( \mathcal{C}_{g,g-2})=3.$$  
\end{prop}
\begin{proof}
From the above discussion we obtain
$${\HH}_{g,g}\setminus\{Y\cup\widetilde{X}\}\cong \mathcal{C}_{g,g-2}\setminus \{X_1\cup X_2\}.$$
As $\codim(\widetilde{X})=\codim(X_1)=\codim(X_2)=2$ we obtain the result.
\end{proof}
This is the final ingredient needed to complete the remaining case.
\begin{prop}
$\Pic_{\QQ}(\Hdg)$ for $d=g$ is trivial.
\end{prop}
\begin{proof}
The three boundary components of $\overline{\HH}_{g,g}\setminus Y$ are linearly independent by Patel~\cite{Patel}. Hence $A^1_\QQ({\HH}_{g,g})\cong\Pic_{\QQ}({\HH}_{g,g})=0$.
\end{proof}
 This provides the final case of Theorem~\ref{thm:PicRank}. 

\bibliographystyle{plain}
\bibliography{base}
\end{document}